\documentclass[12pt,a4paper]{amsart}

\usepackage{amsmath,amssymb,amsthm}
\usepackage{enumitem}
\usepackage{ifthen,verbatim}
\usepackage{mathrsfs}
\usepackage{color}
\usepackage{xcolor}
\usepackage{url}
\usepackage[english]{babel}
\usepackage{here}
\usepackage[T1]{fontenc}
\usepackage{floatflt,graphicx,graphics}
\usepackage{a4wide}
\usepackage{subfigure}
\usepackage{color,currfile}

\usepackage{epstopdf}

\DeclareGraphicsExtensions{.eps,.pdf,.png,.jpg}
\usepackage{color,xcolor}
\usepackage{mathrsfs}

\usepackage{url}

\usepackage{enumerate}
\usepackage{a4wide}
\usepackage{amsmath, amssymb}
\usepackage{amsthm}
\usepackage{float}

\DeclareGraphicsExtensions{.eps,.pdf,.png,.jpg}
\usepackage{color,xcolor}
\usepackage{mathrsfs}

\usepackage[urlcolor=blue,colorlinks=true]{hyperref}
\usepackage{url}

\DeclareMathOperator{\am}{am}

\dedicatory{}
\commby{}
\theoremstyle{plain}
\newtheorem{thm}[equation]{Theorem}
\newtheorem{cor}[equation]{Corollary}
\newtheorem{lem}[equation]{Lemma}
\newtheorem{example}[equation]{Example}

\newtheorem{conjecture}[equation]{Conjecture}

\theoremstyle{definition}

\newtheorem{rem}[equation]{Remark}
\theoremstyle{remark}

\newtheorem{nonsec}[equation]{}

\numberwithin{equation}{section}

\newcommand{\beq}{\begin{equation}}
\newcommand{\eeq}{\end{equation}}
\newcommand{\ben}{\begin{enumerate}}
\newcommand{\een}{\end{enumerate}}
\newcommand{\bequu}{\begin{eqnarray*}}
\newcommand{\eequu}{\end{eqnarray*}}
\newcommand{\bequ}{\begin{eqnarray}}
\newcommand{\eequ}{\end{eqnarray}}
\newcommand{\B}{\mathbb{B}^2}
\newcommand{\R}{\mathbb{R}^2}

\renewcommand{\th}{\,\textnormal{th}}
\renewcommand{\Im}{{ \rm Im}\,}
\renewcommand{\Re}{{ \rm Re}\,}

\newcommand{\uhp}{\mathbb{H}}
\renewcommand{\th}{\mathrm{th}\,}

\font\fFt=eusm10 
\font\fFa=eusm7 
\font\fFp=eusm5 
\def\K{\mathchoice
{\hbox{\,\fFt K}}
{\hbox{\,\fFt K}}
{\hbox{\,\fFa K}}
{\hbox{\,\fFp K}}}

\newcommand{\sn}{\mathrm{sn}\,}
\newcommand{\cn}{\mathrm{cn}\,}
\newcommand{\dn}{\mathrm{dn}\,}

\def\Re{\operatorname{Re}}
\def\Im{\operatorname{Im}}

\newcounter{minutes}
\divide\time by 60
\newcounter{hours}
\multiply\time by 60 \addtocounter{minutes}{-\time}
\begin{document}
\thispagestyle{empty}
\def\thefootnote{}

\title[Intrinsic metrics in polygonal domains]{Intrinsic metrics in polygonal domains}
\author[D.~Dautova]{D.~Dautova}
\address{Institute of Mathematics and Mechanics, Kazan Federal University, 420008 Kazan, Russia}
\email{dautovadn@gmail.com}
\author[R.~Kargar]{R.~Kargar}
\address{Department of Mathematics and Statistics, University of Turku, FI-20014 Turku, Finland}
\email{rakarg@utu.fi}
\author[S.~Nasyrov]{S.~Nasyrov}
\address{Institute of Mathematics and Mechanics, Kazan Federal University, 420008 Kazan, Russia}
\email{semen.nasyrov@yandex.ru}
\author[M.~Vuorinen]{M.~Vuorinen}
\address{Department of Mathematics and Statistics, University of Turku, FI-20014 Turku, Finland}
\email{vuorinen@utu.fi}

\begin{abstract}
We study inequalities between the hyperbolic metric and intrinsic metrics
in convex polygonal domains in the complex plane. Special attention is paid to the triangular ratio metric in rectangles.
A local study leads to an investigation of the relationship between the conformal radius at an arbitrary point of a planar domain and the distance of the point to the boundary.
\end{abstract}

\maketitle

\footnotetext{\texttt{{\tiny File:~\jobname .tex, printed: \number\year-
\number\month-\number\day, \thehours.\ifnum\theminutes<10{0}\fi\theminutes}}}
\makeatletter

\makeatother


\section{Introduction}\label{Sec-Introdiction}
During the past few decades there has been considerable interest
in the study of metrics defined in subdomains $G$ of $\mathbb{R}^n, n\ge 2.$
In geometric function theory the most useful metrics are
\emph{intrinsic metrics}. Distances in these metrics between two points
measure not only how far the points are from each other, but also how close
they are to the boundary of the domain. The hyperbolic distance
$\rho_G(x,y),  x,y \in G,$
of a planar domain $G,$ is ideal for this purpose, in particular, because of its conformal
invariance. There are many reason to study metrics: A.
Papadopoulos \cite[pp.42-48]{p} lists twelve different metrics commonly
used in geometric function theory.
We mention two factors motivating this research:
\begin{itemize}
\item[(a)] In the higher dimensional case $n\ge 3$
there is no metric, equally flexible as the hyperbolic metric is in the planar case.
Thus many authors have introduced  metrics of hyperbolic-type which share some but not all
properties of the hyperbolic metric \cite{hkv}.
\item[(b)] In the planar case,
estimating the  hyperbolic distances of a given planar domain is a
difficult problem \cite{bm}. Therefore, finding concrete estimates for the
hyperbolic metric in terms of simpler metric is needed; these estimates
should take into account both the geometric structure and the metric properties
of the boundary.
\end{itemize}

We study here the topic (b), continuing the earlier work \cite{dnrv,r, rv,hkv}.
 In particular, we investigate the case when the
domain is a polygonal domain in the plane and compare the value of the
hyperbolic metric to the triangular ratio metric $s_G$ \cite{hkv},
see Section \ref{Sec-Preliminaries} for the definition.
In the case when the polygonal domain $G$ is a rectangle,
our main result is Theorem \ref{infinites}. We also formulate
a conjecture about the sharp constant in this theorem
(Conjecture \ref{sharpConst}).

It also turns out that the methodology of our proofs leads to the study of the
\emph{conformal radius} of a domain and its connection with the distance to the boundary. In Section~\ref{confradmax} we study some concrete domains and find the maximum of their conformal radii. In Section~\ref{compar} we investigate the ratio $2 d_G(u)/r_G(u)$ for $u\in G$, where $G$ is a polygonal domain, $d_G(u)=\mbox{\rm dist}(u,\partial G)$ and $r_G(u)$ is the conformal radius of $G$ at $u$. Also we prove that the maximum of the ratio is attained on some graph which consists of those points $u\in G$ for which $d_G(u)=|u - z_j|$, $j=1,2$ for at least two distinct points $z_1\neq z_2$, $z_j\in \partial G$ (Theorem~\ref{ngon}).
Then we give some examples of polygonal domains and calculate the maximal value of $2 d_G(u)/r_G(u)$; the obtained values give the lower estimates for the best constant $C=C(G)$ in the inequality
$$
{\th\frac{\rho_G(u,v)}{2}}  \le C{s_G(u,v)},\quad  u, v\in G.
$$



\section{Preliminaries}\label{Sec-Preliminaries}
Let  $\overline{\mathbb{C}}= {\mathbb{C}} \cup \{\infty\}$ be the extended complex number. Also let $\mathbb{B}^2$ be the unit disk and $ \mathbb{H}^2$ be the
upper half plane in the complex plane $\mathbb{C}$.
\begin{nonsec}{\bf M\"obius transformations.}\label{mymob}
A M\"obius transformation is a mapping of the form
$$z \mapsto \frac{az+b}{cz+d}, \quad a,b,c,d,z \in {\mathbb C},
\ ad-bc\neq 0.$$
The special M\"obius transformation \cite{a}
\begin{equation}\label{myTa}
T_a(z) = \frac{z-a}{1- \overline{a}z}, \quad a \in \B \setminus \{0\},
\end{equation}
maps the unit disk $\mathbb{B}^2$ onto itself with
$T_a(a) =0, T_a( \pm a/|a|)= \pm a/|a|.$
\end{nonsec}

Suppose that $G$ is a proper subdomain in $\R$.
Denote  by $d_G(x)$ the Euclidean distance
dist$(x,\partial G)=
\inf\{|x-z|\text{ }:\text{ }z\in\partial G\}$ between the point $x$
and the boundary of $G$.
The triangular ratio metric $s_G:G\times G\to[0,1]$ is defined by \cite{hkv,rv}
\begin{equation}\label{s}
s_G(x,y)=\frac{|x-y|}{\inf_{z\in\partial G}(|x-z|+|z-y|)},\quad x,y \in G.
\end{equation}

Using the above notation, we can also record the formulas for the
hyperbolic metric $\rho_G$ in these domains $G\in\{\uhp^2,\B\}$ as \cite{b,bm},
 \cite[(4.8), p. 52 \& (4.14), p. 55]{hkv}
\begin{align*}
\text{ch}\rho_{\uhp^2}(x,y)&=
1+\frac{|x-y|^2}{2d_{\uhp^2}(x)d_{\uhp^2}(y)},\quad x,y\in\uhp^2,\\
\text{sh}^2\frac{\rho_{\B}(x,y)}{2}&=
\frac{|x-y|^2}{(1-|x|^2)(1-|y|^2)},\quad x,y\in\B.
\end{align*}
Note that the hyperbolic metric enjoys the following
\emph{conformal invariance} property:
If $G$ is a domain and $f:G\to G'=f(G)$ is a conformal mapping, then
\begin{align*}
\rho_G(x,y)=\rho_{G'}(f(x),f(y)).
\end{align*}
Thus, the hyperbolic metric $\rho_G$ can be defined in any planar
simply connected
domain in terms of a conformal mapping of the domain onto
the unit disk \cite[Thm 6.3 p. 26]{bm}. In particular, the
hyperbolic metric is invariant under {M\"obius transformations}.
In terms of complex numbers, the formulas of the  hyperbolic metric for
the two cases $\uhp^2$ and $\B$ can be simplified to
\begin{align*}
\text{th}\frac{\rho_{\uhp^2}(x,y)}{2}&=\text{th}\left(\frac{1}{2}
\log\left(\frac{|x-\overline{y}|+|x-y|}{|x-\overline{y}|-|x-y|}\right)\right)
=\left|\frac{x-y}{x-\overline{y}}\right|,\\
\text{th}\frac{\rho_{\B}(x,y)}{2}&=\text{th}\left(\frac{1}{2}
\log\left(\frac{|1-x\overline{y}|+|x-y|}{|1-x\overline{y}|-|x-y|}\right)\right)
=\left|\frac{x-y}{1-x\overline{y}}\right|,
\end{align*}
where $\overline{y}$ is the complex conjugate of $y$.

\begin{nonsec}{\bf Elliptic integrals and Jacobian elliptic functions.}
Let
$\K(\lambda),0< \lambda< 1,$ denote the complete elliptic integral of the
first kind \cite{avv},
\begin{equation} \label{Kdef}
\K(\lambda)=\int_0^1\frac{{\rm d}t}{\sqrt{(1-t^2)(1-
\lambda^2 t^2)}}, \quad \K'(\lambda)=\K(\sqrt{1-\lambda^2}).
\end{equation}
Let $\sn(\cdot,\lambda)$,
$\cn(\cdot,\lambda)$ and $\dn(\cdot,\lambda)$
denote the Jacobi elliptic functions. We recall that, by definition,
$\sn(z,\lambda)=\sin \am(z,\lambda)$, $\cn(z,\lambda)=\cos \am(z,\lambda)$  and $\dn(z,\lambda)=\frac{d}{dz}\am(z,\lambda)$ where $\am(z,\lambda)$ is the Jacobi amplitude, i.e. the function $\varphi=\am(z,\lambda)$ inverse to the incomplete elliptic integral of the
first kind \cite{af}
$$
z=\int_0^\varphi\frac{{\rm d} \theta}{\sqrt{1-\lambda^2\sin^2\theta}}.
$$
Elliptic integrals and elliptic functions occur in conformal mapping
of a rectangle onto the upper half plane \cite[p. 358]{af}.

\end{nonsec}

\begin{nonsec}{\bf Conformal radius.}
The \emph{conformal radius} $r_D(z_0)$ of a simply connected domain $D$
with a nondegenerate boundary at a point $z_0\neq \infty$ is a
well-known characteristic playing an important role in geometric
function theory and applications. By definition, $r_D(z_0)$ is the
radius of a disk centered at the origin which can be mapped onto $D$
conformally by a function $F$  with the normalization $F(0)=z_0$ and $F'(0)>0$
 (see, e.g. \cite{gol,ps,bf,bkn,sol}).
Let $g:D\to \mathbb{B}^2$ be a conformal mapping of $D$ onto
$\mathbb{B}^2$ and $f$ be its inverse. If $g(z_0)=\zeta$, then \cite[p. 162]{ps}
\begin{equation}\label{crad1}
r_D(z_0)=|f'(\zeta)|(1-|\zeta|^2)=\frac{1-|g(z_0)|^2}{|g'(z_0)|}
\end{equation}
and
if $G:D\to \mathbb{H}^2$ is a conformal mapping of $D$ onto the upper
 half plane, $F$ is its  inverse and $G(z_0)=w$, then
\begin{equation}\label{crad2}
r_D(z_0)=2\Im w\,|F'(w)|=\frac{2\Im G(z_0)}{|G'(z_0)|}.
\end{equation}

It is well known that the conformal radius $r_D(z_0)$ and the distance to the boundary $d_D(z_0)$ are equivalent in the sense that for every simply connected domain $D$
\begin{equation}\label{rd4}
d_D(z_0)\le r_D(z_0)\le 4 d_D(z_0).
\end{equation}
The left estimation follows from the Schwarz lemma and the right one is a corollary of the Koebe quarter theorem. It should be noted that both inequalities \eqref{rd4} above are sharp. The first inequality is sharp when $f(z)=z$ and $D={\mathbb{B}^2}$. Also the Koebe function $k(z):=z/(1-z)^2$ exhibits equality in the second inequality. Since $k(z)$ maps the open unit disk ${\mathbb{B}^2}$ onto $D:=\mathbb{C}\setminus (-\infty,-1/4]$, we get $r_D(0)=1=4d_D(0)$.

Moreover, for a convex domain $D$ and an arbitrary point  $z_0\in D$
\begin{equation}\label{rd2}
d_D(z_0)\le r_D(z_0)\le 2 d_D(z_0)
\end{equation}
and both the inequalities are sharp; in the first inequality, by the Schwarz lemma,
the equality holds if and only if $D$ is a disk and $z$ is
its center; in the second inequality we have equality if and
only if $D$ is a half-plane and $z$ is its arbitrary point (see, e.g. \cite[Ch.2, \S~2.5, Thrm~2.15]{duren}, \cite[Ch.1, Sect~1.4, Thrm~1.8]{hayman}).
\end{nonsec}



\section{Nearest point for a rectangle}

Consider the rectangle $R$ with vertices $\pm k\pm i$, where $k\ge1$.
From the definition \eqref{s} of $s$-metric we see that if
a point $x\in R$ is fixed then the infimum
$\inf_{z\in\partial G}(|x-z|+|z-y|)$ is attained at some points $z$
lying on a side of $\partial R$.  This side depends on the location of $y$.

In Fig. \ref{near} we show, for a fixed $x \in R,$ the subdomains of $R$ such that if $y$
belongs to one of these subdomains, then  $z$ belongs to the
corresponding side of $R$ which is a common side of $R$ and of the subdomain.
Among the four  subdomains, two are
trapezoids and other two are triangles. The lower and the upper
subdomains are trapezoids if and only if $\Re x^2-\Im x^2< k^2-1$;
in the case $\Re x^2-\Im x^2> k^2-1$ trapezoids are the left and
right subdomains. The case $\Re x^2-\Im x^2= k^2-1$ corresponds to a
subdivision of $R$ into four triangles. In Fig. \ref{near} we show
a subdivision for $k=1.4$, $x=0.7-0.4i$.

\begin{figure}[ht] \centering
\includegraphics[width=3.in]{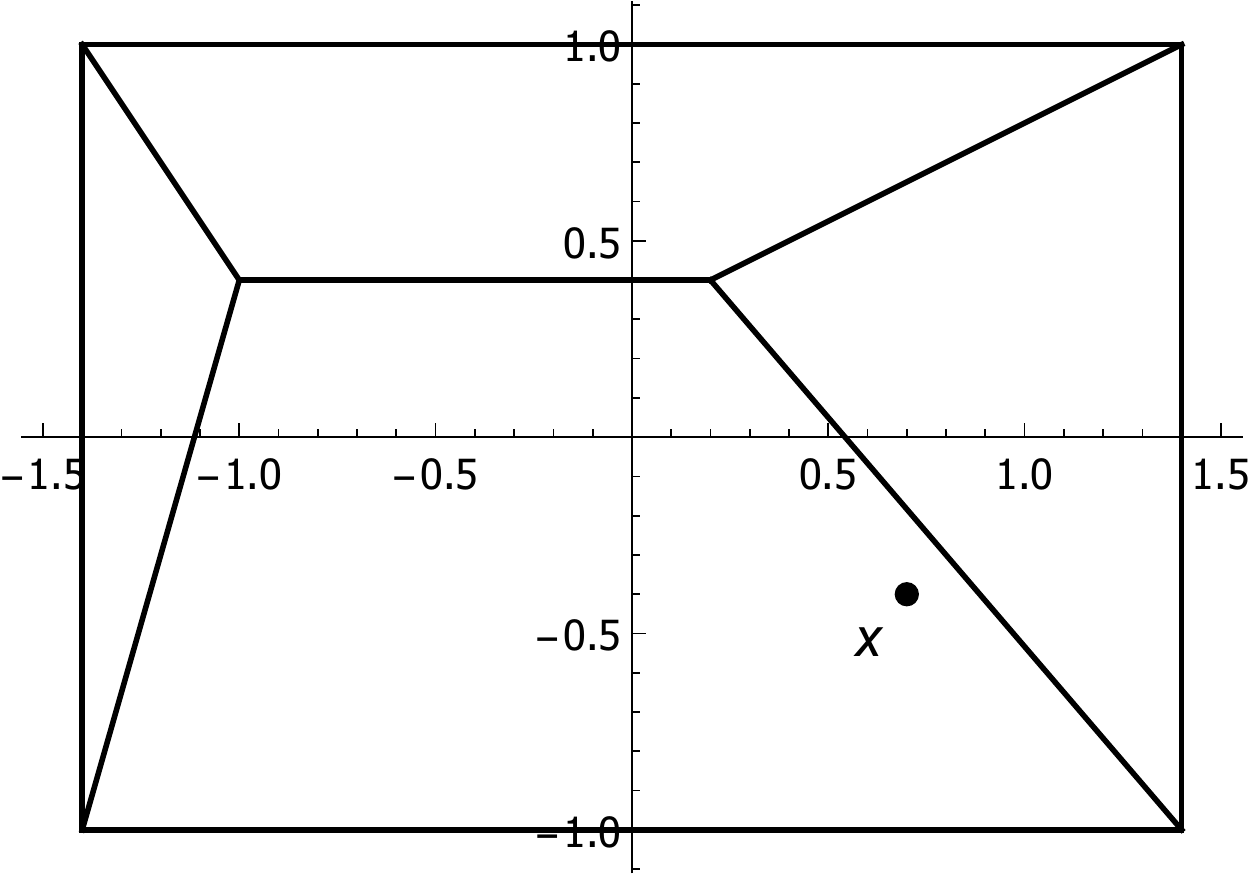}
\caption{For a fixed point  $x$ in a rectangle, there are four subdomains
with the following property.
If $y$ is a point of a subdomain, then  the  point $z$ providing the minimum
of $|x-z|+|z-y|$ is on the common side
 of the subdomain and of the rectangle.}
\label{near}
\end{figure}


We can also describe the geometry of disks and circles in the $s$-metric.
Every circle centered at $x$ of radius $r$ is either a Euclidean circle
or a piecewise smooth Jordan curve.
In Fig. \ref{levels} we show such curves for the case  $k=1.4$,
$x=0.7-0.4i$, $r=0.1j$, $1\le j\le 9$.  We note that the corner points
of these curves are on the segments separating $R$ into four parts
described above (see Fig. \ref{near}). Besides, the disks are convex
sets in the Euclidean metric.

\begin{figure}[ht] \centering
\includegraphics[width=3.5 in]{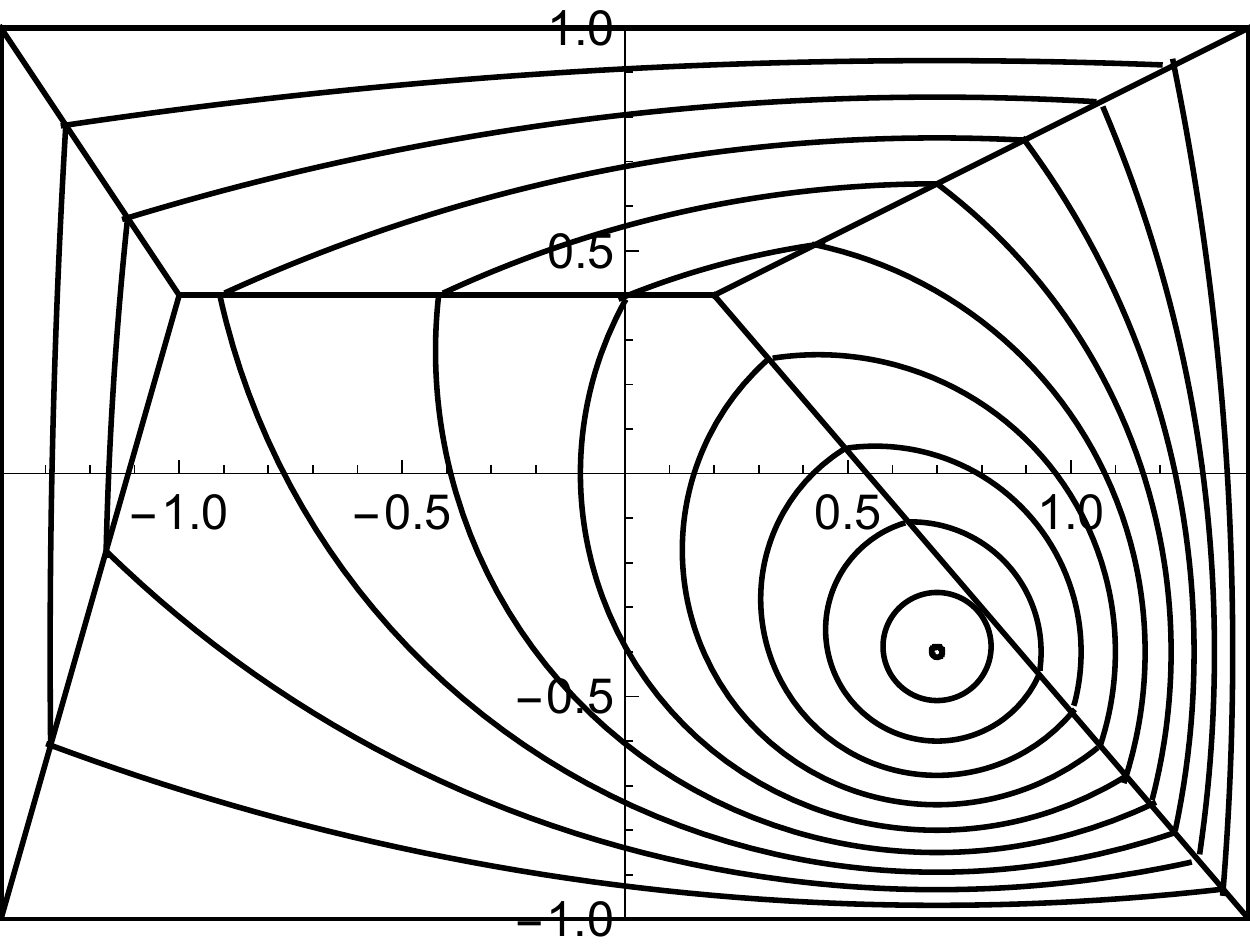}
\caption{Some $s$-metric circles in the rectangle
$R=[-1.4,1.4]\times[-1,1]$ centered
at $x=0.7-0.4i$ with radii $r=0.1j$, $1\le j\le 9$.}\label{levels}
\end{figure}



\section{Hyperbolic metric versus triangle ratio metric in rectangle}\label{rectht}


Our goal is to find for the rectangle $R$ the best constants $C_1,C_2>0$ such that  the inequality
\begin{equation}\label{C12}
C_1{s_R(u,v)}\le{\th\frac{\rho_R(u,v)}{2}}  \le C_2{s_R(u,v)},\quad  u, v\in R
\end{equation}
holds.
Now we will find the limit
$$
\lim_{u\to v}\frac{\th\frac{\rho_R(u,v)}{2}}{s_R(u,v)}
$$
and give its geometric interpretation.

Consider the segments $\Sigma_1$ and $\Sigma_2$ which are parts of the
bisectors of the angles at the  vertices $-k-i$ and $-k+i$ of $R$, respectively,
connecting these vertices with the
point $-(k-1)$ on the real axis;  similarly, let also $\Sigma_3$ and
$\Sigma_4$ join the vertices $k-i$ and $k+i$   with the point $k-1$ on the real axis.
Consider also the segment $\Sigma_0$ with endpoints $\pm(k-1)$.
These five segments subdivide $R$ into four subdomains, two trapezoids
and two triangles (see Fig. \ref{separate});  in the case
$k=1$ the trapezoids degenerate into triangles and $\Sigma_0$
degenerates into a point.

\begin{figure}[ht] \centering
\includegraphics[width=4. in,%
]{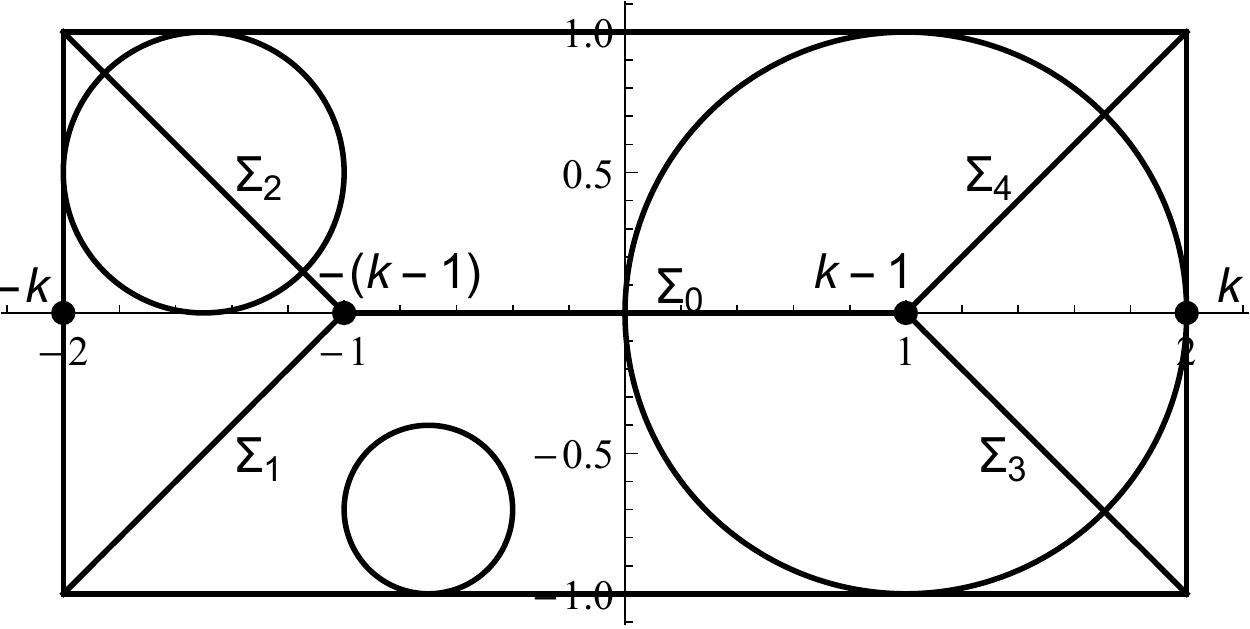}\
\caption{Decomposition of a rectangle $R$ and the five segments $\Sigma_j, 0\le j \le 4$, $k=2$.}
\label{separate}
\end{figure}

We see that if a point $v$ is in one of these subdomains, then the point
$\zeta_0$, where the minimal value of $|\zeta-v|$, $\zeta\in \partial R$,
is attained, lies on the side of $R$ which is common with this subdomain. If $v$
is on one of the segments described above and $v\neq\pm (k-1)$, then
the shortest distance is attained at two sides of $R$. At last, if
$v=\pm(k-1)$  then there are three sides where the minimal value of
$|\zeta-v|$, $\zeta\in \partial R$, is attained; in the case $k=1$, we have the square, the points $\pm(k-1)$ coincide with its center and the minimal distance is attained at all four the sides.

We can describe this situation as follows. If  $v$ is an interior point
of one subregion, then the largest disk contained in $R$ and centered at $v$
touches $\partial R$ at one point lying on the  of $R$ which is common with
this subdomain. If  $v\in \cup_{j=0}^4\Sigma_j$ and $v\neq\pm (k-1)$, then
the largest disk  touches two of the sides, and if
$v=\pm(k-1)$, then it touches $\partial R$ at three points. These
three case are shown on Fig. \ref{separate}. (If $k=1$, then the
largest disk touches all  the four sides of $R$.)

\begin{lem}\label{limrhos}
We have
\begin{equation}\label{rat}
\lim_{u\to v}\frac{\th\frac{\rho_R(u,v)}{2}}{s_R(u,v)}=\frac{2 d_R(v)}{r_R(v)},
\end{equation}
where $d_R(v)=\mbox{\rm dist}(v,\partial R)$ and $r_R(v)$ is the conformal radius of $R$ at the point $v$.
\end{lem}

\begin{proof}
Let $f:R\to\uhp^2$ be a conformal mapping of $R$ onto the upper half plane $\uhp^2$. Then, by the definition,
$$
\th\frac{\rho_R(u,v)}{2}=\frac{|f(u)-f(v)|}{|f(u)-\overline{f(v)}|}.
$$
Let $v$ be a point from the lower trapezoidal part of $R$ (other three cases can be considered similarly).
Then
$$
s_R(u,v)=\frac{|f(u)-\overline{f(v)}|}{|u-\overline{v}+2i|}.
$$
Therefore,
$$
\frac{\th\frac{\rho_R(u,v)}{2}}{s_R(u,v)}=\frac{|f(u)-f(v)|}{|u-v|}
\frac{|u-\overline{v}+2i|}{|f(u)-\overline{f(v)}|}\to 2(\Im (v)+1)\frac{|f'(v)|}{2\Im f(v)}.
$$
Since
$$
\Im (v)+1=\mbox{\rm dist}(v,\partial R) \quad \text{and} \quad \frac{|f'(v)|}{2\Im f(v)}=r_R(v),
$$
we obtain \eqref{rat}.
\end{proof}
We see that if \eqref{C12} holds, then by Lemma \ref{limrhos} we get
$$
\frac{2}{C_2} d_R(v)\le r_R(v)\le \frac{2}{C_1}d_R(v).
$$
Therefore, we obtain an  estimation of the conformal radius $r_R(v)$ through $d_R(v)$. On the other hand, if for some point $v\in R$ we have
$$
r_R(v)=\widetilde{C}d_R(v),\quad 0<\widetilde{C}\le 1,
$$
then the constant $C_2$ in \eqref{C12} satisfies $C_2\ge 2/\widetilde{C}$.
From \eqref{rd4} we obtain
$$
\frac{1}{2}\le\frac{2 d_D(v)}{r_D(v)}\le 2, \quad v\in D
$$
and for convex domains, with the help of
\eqref{rd2}, we deduce that
$$
1\le\frac{2 d_D(v)}{r_D(v)}\le 2, \quad v\in D.
$$
Given $k\ge 1$, let $\lambda\in(0,1)$
be the unique root of the equation
\begin{equation} \label{lambda}
\frac{\K(\lambda)}{\K'(\lambda)}=\frac{k}{2},
\end{equation}
where $\K(\lambda)$ is as defined in \eqref{Kdef}.


\begin{lem}\label{segm}
Let $D$ be a convex planar domain with nonempty boundary $\partial D$ and
suppose that one of the following conditions is valid:

(i) The circle centered at a point $z_0\in D$ of radius $d_D(z_0)$ touches $\partial D$ at some point $\zeta_0$;

(ii)  The boundary $\partial D$ contains two linear segments $[a,\zeta_0]$ and $[\zeta_0,b]$
and the circle centered at a point $z_0\in D$ of radius $d_D(z_0)$ touches $\partial D$ at two points $\zeta_1\in [a,\zeta_0]$ and $\zeta_2\in [\zeta_0,b]$.

Then for every $z_1$ which is an interior point of the segment with endpoints $z_0$, $\zeta_0$ we have
\begin{equation}\label{dr}
\frac{d_D(z_1)}{r_D(z_1)}\le\frac{d_D(z_0)}{r_D(z_0)}.
\end{equation}
Moreover, if either (i) holds and $D$ is distinct from a half plane, or (ii) is valid and $\partial D$ is distinct from the union of two ray coming from the point $\zeta_0$ and passing through the points $a$ and $b$, then the inequality in \eqref{dr} is strict.
\end{lem}

\begin{proof}Since  $d_D$ and $r_D$ are invariant under the shifts of the plane, without loss of generality we can assume that  $\zeta_0=0$.
Let $\alpha=z_0/z_1$, $\alpha>1$. Denote $G=\psi(D)$ where
the mapping $\psi:z\mapsto \alpha z$.
We note that $\psi(z_1)=z_0$ and  $r_G(z_0)=\alpha r_D(z_1)$,
$d_G(z_0)=\alpha d_D(z_1)$, therefore,
$d_D(z_1)/r_D(z_1)=d_G(z_0)/r_G(z_0)$.
But because $D$ is convex, we obtain $D\subset G$, therefore,
$r_G(z_0)\ge r_D(z_0)$. Since evidently  $d_G(z_0)= d_D(z_0)$,
we obtain \eqref{dr}.
If either (i) holds and $D$ is distinct from a half plane, or (ii) is valid and $\partial D$ is distinct from the union of two ray coming from the point $\zeta_0$ and passing through the points $a$ and $b$, then $D\neq G$ and,
therefore,   $r_G(z_0)< r_D(z_0)$, so the inequality in \eqref{dr} is strict.
\end{proof}

\begin{thm}\label{infinites}
If $R$ is the  rectangle $[-k,k]\times[-1,1]$, $k\ge 1$, then for all $v \in R$
\begin{equation}\label{2dr}
1< \frac{2 d_R(v)}{r_R(v)}\le C(\lambda),
\end{equation}
where
\begin{equation}\label{Cl}
C(\lambda)=
\K(\lambda)\left|\frac{\cn(i\K(\lambda),\lambda)\dn(i\K(\lambda),
\lambda)}{\sn(i\K(\lambda),\lambda)}\right|.
\end{equation}
Both inequalities of \eqref{2dr} are sharp.
Moreover, equality in the second inequality holds if and only if $v=\pm (k-1)$.
\end{thm}

\begin{proof}
The lower estimation in \eqref{2dr} is evident. We will prove the upper one.
First we will show that the maximal value of $2 d_R(v)/r_R(v)$, $v\in R$,
is attained only on the segment with endpoints $\pm (v-1)$.
Using the statement  (i) of Lemma~\ref{segm}, we see that the
maximum of the value $2 d_R(v)/r_R(v)$, $v\in R$, is attained if
$v\in \cup_{j=0}^4\Sigma_j$. By the statement  (ii)
of Lemma~\ref{segm} we conclude that $v$ is on the segment
$\Sigma_0$ with endpoints $\pm(k-1)$.

For  $v \in \Sigma_0$ we have $d_{R}(v)= 1$.
Therefore, to find the maximal value of $2 d_R(v)/r_R(v)$ we
only need to find the minimal value of the conformal radius $r_R(v)$,
$v\in \Sigma_0$.
But $r_R(v)$ is an even function on $[-k,k]$ which is strictly decreasing on $[0,k]$. Therefore, the minimal value of $r_R(v)$ on $\Sigma_0=[-(k-1),k-1]$ is attained at the endpoints of the segment, i.e. at the points $\pm (k-1)$.

Now we will find the value $r_R(k-1)$. To calculate it we replace,
for convenience,
 $R$ with its image $R_2=[-\alpha,\alpha]\times [0,2\alpha k]$ under
the mapping $z\mapsto i\alpha(k-z)$ where the constant $\alpha>0$ will be
fixed below; we note that, under the mapping, the the point $k-1$ goes to
$i\alpha$. This mapping does not change the value
of the conformal radius.
It is known that $R_2$ can be mapped onto the upper half-plane
$\uhp^2$ such that its vertices $\alpha(-1+2ik)$, $-\alpha$,
$\alpha$, and   $\alpha(1+2ik)$ will go to $-1/\lambda$, $-1$, $1$, $1/\lambda$;
here $\lambda\in(0,1)$ is the unique root of the equation \eqref{lambda}.
If we put $\alpha=\K(\lambda)$, then we can use as such a mapping
the Jacobi elliptic function $f(z)=\sn(z,\lambda)$ which maps $R_2$ onto
$\uhp^2$ with the correspondence of boundary points indicated above \cite[p.358]{af}.
By the definition of the conformal radius, we have
$$
r_{R_2}(i\K(\lambda))=
\frac{2\Im f(i\K(\lambda))}{|f'(i\K(\lambda))|}.
$$
Since $f'(z)=\cn(z,\lambda)\dn(z,\lambda)$ and $f(i\K(\lambda))$ is a pure imaginary number, we find
$$
r_R(k-1)=r_{R_2}(i\K(\lambda))=\frac{2\sn(i\K(\lambda),\lambda)|}{|\cn(i\K(\lambda),\lambda)
\dn(i\K(\lambda),\lambda)|}.
$$
We also have $d_{R_2}(i\K(\lambda))=\K(\lambda)$. Since the ratio of the distance to the boundary and the conformal radius is invariant under linear conformal mappings, we have
$$
\frac{2d_R(k-1)}{r_R(k-1)}=\K(\lambda)\left|\frac{\cn(i\K(\lambda),\lambda
)\dn(i\K(\lambda),\lambda)}{\sn(i\K(\lambda),\lambda)}\right|.
$$
As we showed above, this is the biggest value of $2d_R(v)/r_R(v)$ over $R$.
\end{proof}
Since $\sn(i\K(\lambda),\lambda)$ is a pure imaginary complex number,
$\cn(z,\lambda)=\sqrt{1-\sn^2(z,\lambda)}$ and
$\cn(z,\lambda)=\sqrt{1-\lambda^2\sn^2(z,\lambda)}$,
we can write \eqref{Cl} in the form
$$
C(\lambda)=\K(\lambda)\frac{\sqrt{(1+a^2(\lambda))(1+\lambda^2a^2(\lambda))}}
{a(\lambda)},\quad {\rm where}\quad  a(\lambda):=|\sn(i\K(\lambda),\lambda)|.
$$
\begin{cor} \label{sharpConst} We have
$$
\frac{2}{C(\lambda)}d_R(v)\le{r_R(v)}< 2 d_R(v),\quad v
\in R
$$
and the inequalities are sharp.
\end{cor}
\begin{conjecture}\label{conclambda}
Let $C(\lambda)$ be defined as in \eqref{Cl}. Then for all points $u$, $v$ lying in the rectangle $R=[-k,k]\times[-1,1]$, $k\ge 1$, we have
\begin{equation}\label{thrl}
{s_R(u,v)}\le{\th\frac{\rho_R(u,v)}{2}}  \le C(\lambda){s_R(u,v)}.
\end{equation}
These inequalities are sharp.
\end{conjecture}

\begin{rem}\label{rem1}
The inequality $s_R(u,v)\le \th(\rho_R(u,v)/2)$ and its sharpness is evident. \medskip
\end{rem}

Above we considered rectangles $R=[-k,k]\times [-1,1]$ with $k\ge 1$.
This corresponds to the values
$\lambda\ge \lambda_0:=3-2\sqrt{2}=0.171572875\ldots$, where $\lambda_0$ is the unique root of the equation \eqref{lambda} for $k=1$ \cite[p.81, Table 5.1]{avv}.
We can also consider a rectangle with $k<1$. It is evident that this
case is easily reduced to the case $k\ge 1$ because we can apply
the mapping $z\mapsto (i/k)z$ which maps $R$ onto the rectangle
$[-1/k,1/k]\times [-1,1]$. For such a rectangle
we replace $\lambda$ with
 $h(\lambda)=(1-\sqrt{\lambda})^2/(1+\sqrt{\lambda})^2$.
Therefore, let us define $\widetilde{C}(\lambda)$ as follows:
 \begin{equation*}
   \widetilde{C}(\lambda):=\left\{
      \begin{array}{ll}
        C(h(\lambda)), & \hbox{$0<\lambda<\lambda_0$;} \\[2mm]
        C(\lambda), & \hbox{$\lambda_0\le \lambda<1$.}
      \end{array}
    \right.
 \end{equation*}
Then
$${r_R(v)}\le 2d_R(v)  \le \widetilde{C}(\lambda){s_R(u,v)}.$$
The graph of $\widetilde{C}(\lambda)$ for $0<\lambda<1$ is shown in Fig. \ref{grphiC}.

\begin{figure}[ht] \centering
\includegraphics[width=3. in,%
]{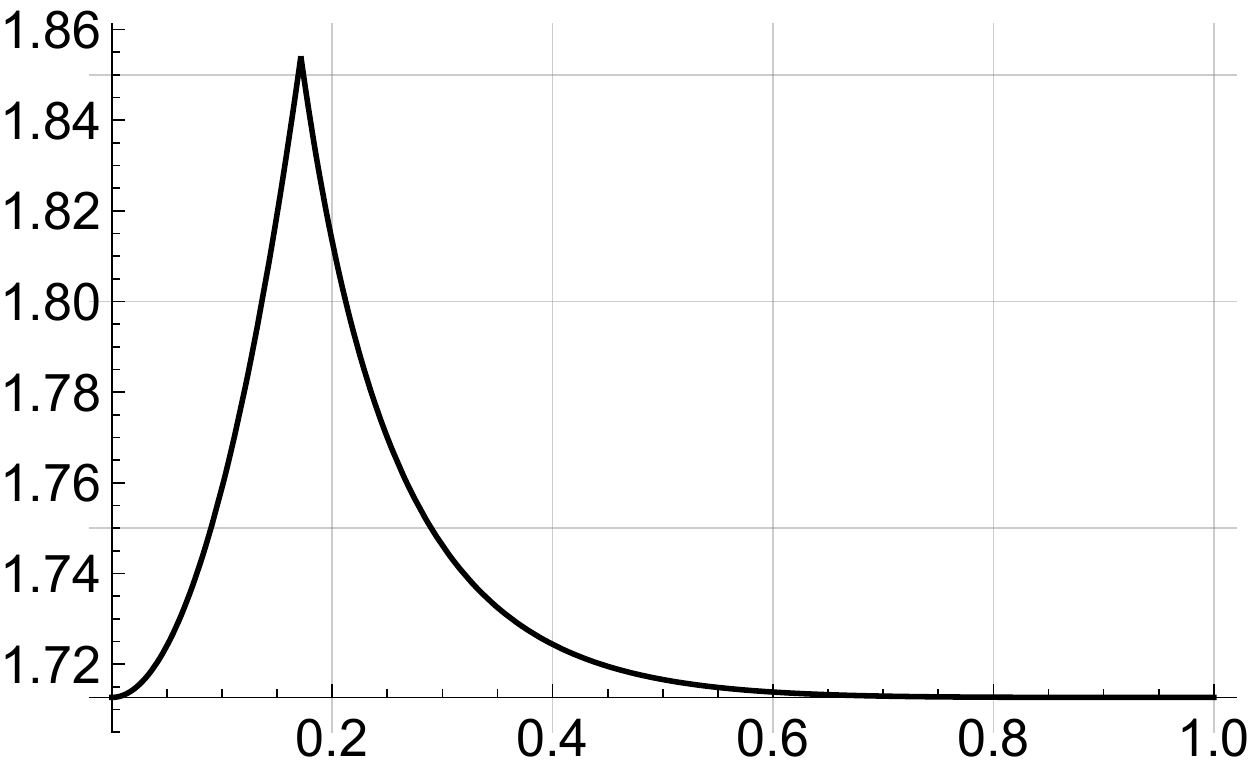}\
\caption{The graph of the function $\widetilde{C}(\lambda)$.}\label{grphiC}
\end{figure}

We have
\begin{equation}\label{l0}
\widetilde{C}(\lambda_0)=\varphi(\lambda_0):=\frac{(1+\lambda_0)\K'(\lambda_0)}{2}    =1.854074677\ldots,
\end{equation}
$$
\lim_{\lambda\to 0+}\widetilde{C}(\lambda)=\lim_{\lambda\to 1-}\widetilde{C}(\lambda)=\frac{\pi}{2}\coth\frac{\pi}{2}=1.7126885749\ldots.
$$
The first limiting case corresponds to a square and the second
one matches to the case of a half-strip.


\section{ Conformal radius of specific domains and its maximum}\label{confradmax}

In this section, we determine the maximal value of the conformal radius
for some domains which are conformal images of the unit disk $\mathbb{B}^2$.
  These domains can be found in \cite{kt2021} and its references.
We note that exact values of the maximal value for conformal radius of some special domains can be found in \cite[Tables of some functionals]{ps}.

\begin{example}\label{ex1}\rm
  Consider the univalent function $w=\phi_1(z)=\sqrt{1+z}-1$, $z\in {\mathbb{B}^2}$, where the branch of the square root is chosen such that $\sqrt{1}=1$. The function $\phi_1$ maps the open unit disk ${\mathbb{B}^2}$ onto the domain $D_1$ which is the right-half of the lemniscate of Bernoulli. Hence, the inverse mapping $z=\varphi_1(w)=w(w+2)$ maps $D_1$ onto ${\mathbb{B}^2}$.
  Since $\varphi_1(0)=0$ and $\varphi_1'(0)=2>0$, we have $r_{\phi_1({\mathbb{B}^2})}(0)=0.5$.

Now we will find the maximum of the conformal radius of $D_1$. If $w=\phi_1(z)$, where $z=t e^{i\theta}$, $t\in[0,1)$ and $\theta\in[0,2\pi)$, then by  \eqref{crad1} the conformal radius of $D_1$ at the point $w$ is equal to
$$
r_{D_1}(w)=\frac{1}{2}|1+z|^{-1/2}(1-|z|^2)\le\frac{1}{2}(1-t)^{-1/2}(1-t^2)
=\frac{1}{2}(1-t)^{1/2}(1+t)
$$
and the equality holds if and only if $\theta=\pi$. The maximal value of the function $(1-t)^{1/2}(1+t)/2$  is attained at $t_0=1/3$ and is equal to $2\sqrt{6}/9$.
Thus we see that the maximum of $r_{D_1}(w)$ in ${\mathbb{B}^2}$ is equal to $2\sqrt{6}/9=0.544331\ldots$ which is attained at the point $w=\phi_1(-1/3)=\sqrt{6}/3-1=-0.183503\ldots$.
\end{example}

\begin{example}\label{ex2}\rm
  Consider the univalent function $w=\phi_2(z)=z+\sqrt{1+z^2}-1$, $z\in {\mathbb{B}^2}$, where the branch of the square root is chosen such that $\sqrt{1}=1$. We note that this function is closely connected with the function inverse to the Joukowsky transform.  It maps the unit disk ${\mathbb{B}^2}$ onto the crescent-shaped region $D_2$, see Fig. \ref{fig:subfig1}.
  The inverse function $z=\varphi_2(w)=(w+1-(1/(1+w)))/2$ maps $D_2$ onto ${\mathbb{B}^2}$. Hence, $\varphi_2(0)=0$, $\varphi_2'(0)=1$,  therefore $r_{D_2}(0)=1$.

  Now we will find the maximal value of the conformal radius $r_{D_2}(w)$ for points $w$ lying on the real axis. Let $w=\phi_2(t)$, $t\in(-1,1)$, then, with the help of \eqref{crad1}, we obtain
  $$
  r_{D_2}(w)=\left|1+\frac{t}{\sqrt{1+t^2}}\right|(1-|z|^2)\le \left(1+\frac{|t|}{\sqrt{1+t^2}}\right)(1-t^2).
  $$
  Simple analysis shows that the maximal value of the expression in the right-hand side equals $M=1.17117\ldots$, it is attained at the point $t_0=0.319968\ldots$ which is the minimal positive root  of the equation $4t^8+12 t^6+t^4-10t^2+1=0$. The corresponding point of maximum of the conformal radius is equal to $w=\phi_2(t_0)=0.369911\ldots$. Therefore, we found that $\max_{-1<w<1} r_{D_2}(w)=M$. Finding the maximum of $r_{D_2}(w)$ over all $w\in D_2$ is an open problem.
\end{example}

\begin{example}\rm
  Consider the analytic function $w=\phi_3(z)=\alpha z+\beta z^2$, $z\in {\mathbb{B}^2}$, $\alpha,\beta\in\mathbb{C}$ and $|\alpha|\ge 2|\beta|>0$. This function maps the unit disk ${\mathbb{B}^2}$ onto the domain $D_3$ bounded by a heart-shaped curve, see Fig. \ref{fig:subfig2} for $\alpha=4/3$ and $\beta=2/3$. The maximal value of the conformal radius $r_{D_3}(w)$ equals $\ell(r_1)$, where
 \begin{equation*}
  \ell(r):=|\alpha|+2|\beta|r-|\alpha|r^2-2|\beta|r^3
 \end{equation*}
 and
  \begin{equation*}
    r_1=\frac{-|\alpha|+\sqrt{|\alpha|^2+12|\beta|^2}}{6|\beta|};
  \end{equation*}
the maximal value is attained at the point $w_0=\phi_3(r_1e^{i\delta})$ where $\delta=\arg(\alpha/\beta)$. For example, if we let $\alpha=\sqrt{2}$ and $\beta=1/2$, then the function $w=\phi_4(z)=\sqrt{2}z+z^2/2$, $z\in {\mathbb{B}^2}$ maps the unit disk ${\mathbb{B}^2}$ onto the domain $D_4$ bounded by a lima\c{c}on, see Fig. \ref{fig:subfig3}.
  The maximal value of the conformal radius $r_{D_4}(w)$ equals  $(2/27) (7\sqrt{2} + 5\sqrt{5})=1.56147\ldots$, it is attained at the point $w_0=\phi_4((\sqrt{5}-\sqrt{2})/3)=(4\sqrt{10}-5)/18=0.424951\ldots$.
\end{example}



\begin{figure}[!ht]
\centering
\subfigure[]{
\includegraphics[width=4cm]{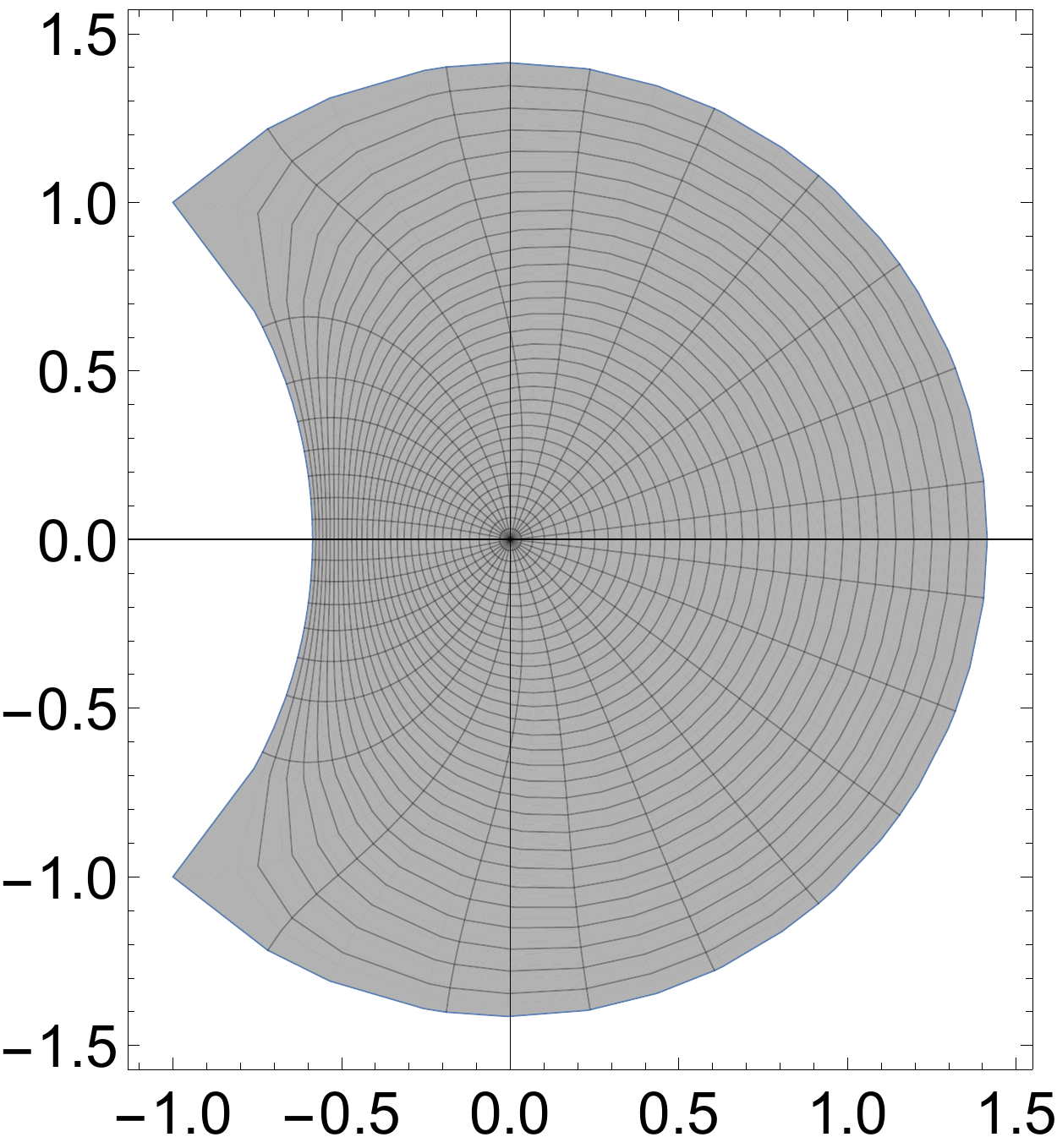} 
    \label{fig:subfig1}
}
\hspace*{7mm}
\subfigure[]{
\includegraphics[width=3.95cm]{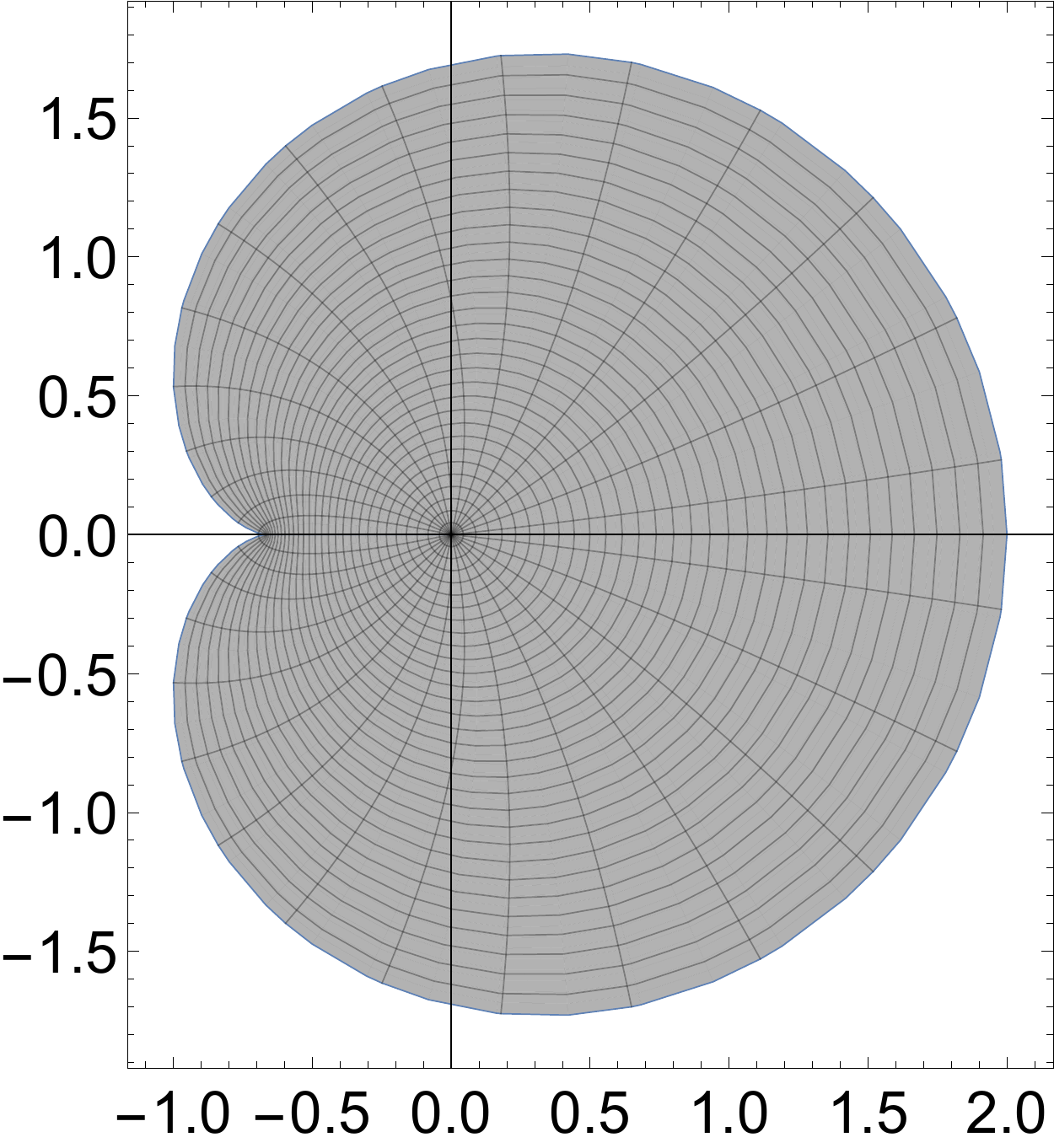} 
    \label{fig:subfig2}
}
\hspace*{7mm}
\subfigure[]{
\includegraphics[width=4.04cm]{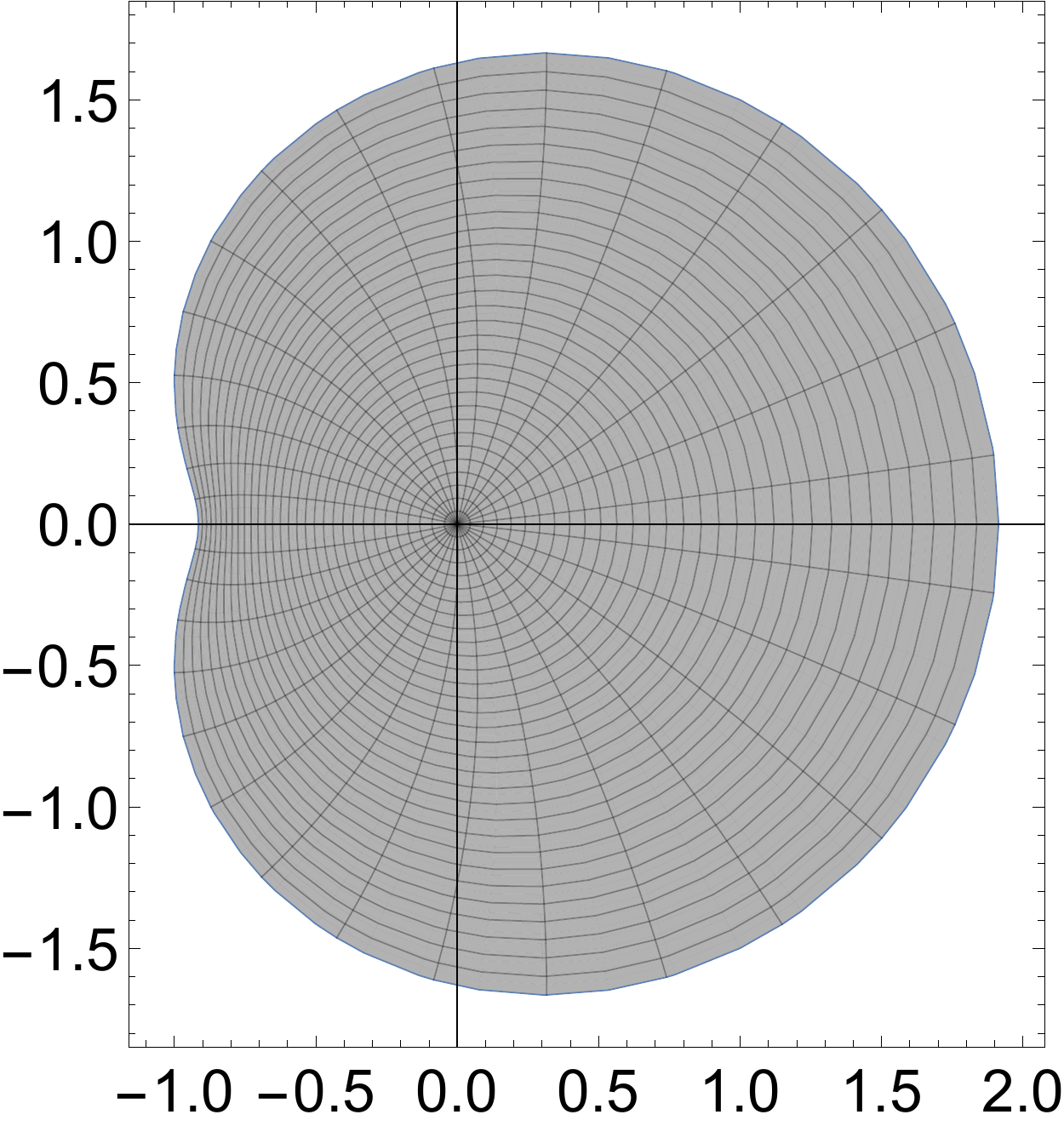} 
    \label{fig:subfig3}
}
\caption[The boundary curve of $\phi_j(\mathbb{B}^2)$, $j=2,3,4$]
{\subref{fig:subfig1}: The graph of $\phi_2(\mathbb{B}^2)$
 \subref{fig:subfig2}: The graph of $\phi_3(\mathbb{B}^2)$ for $\alpha=4/3$ and $\beta=2/3$
\subref{fig:subfig3}: The graph of $\phi_4(\mathbb{B}^2)$}
\end{figure}

\begin{example}\rm
Consider the univalent function $w=\phi_5(z)=z/(1-\alpha z^2)$, $z\in {\mathbb{B}^2}$, $0\leq \alpha<1$. This function $\phi_5$ maps the unit disk ${\mathbb{B}^2}$ onto the domain $D_5$, bounded by a Booth lemniscate.
We have
$$
\phi_5'(z)=\frac{1+\alpha z^2}{(1-\alpha z^2)^2},
$$
therefore,  for $w=\phi_5(z)$ we have
$$
r_{D_5}(w)=\left|\frac{1+\alpha z^2}{(1-\alpha z^2)^2}\right|(1-|z|^2)\le \psi(t^2):=\frac{1+\alpha t^2}{(1-\alpha t^2)^2}(1-t^2),\quad t=|z|.
$$
Now we will investigate the behavior of $\psi(\tau)$, $\tau\in [0,1]$.
Simple analysis shows that for $0\le \alpha\le 1/3$ the function  $\psi$ decreases on [0,1]. Therefore,  $\psi(t^2)\le \psi(0)=1$ and the maximal value of the conformal radius  $r_{D_5}(w)$ equals $1$; it is attained at the origin.

If $1/3<\alpha<1$, then $\psi$ has a unique point of maximum $\tau_\alpha=(3\alpha-1)/(\alpha(3-\alpha))$. Thus, in this case, the maximal value $m(\alpha)$ of the conformal radius  $r_{D_5}(w)$ equals
$\psi(\tau_\alpha)$, i.e.
$$
m(\alpha)=\frac{(1+\alpha)^2}{8\alpha(1-\alpha)}.
$$
It is attained at the points
$$
w_0=\pm\phi_5(\tau^2_\alpha)=\pm \frac{(3\alpha-1)(3-\alpha)}{(1-\alpha)(-\alpha^2+14\alpha-1)}.
$$

\end{example}


\section{Comparison of the conformal radius and the distance to the boundary for some domains}\label{compar}

It is easy to see that we can apply the above method, based on Lemma~\ref{segm}, to obtain sharp two-sided estimations of $2d_D(\cdot)/r_D(\cdot)$ for other convex planar domains $D$.

First consider  a bounded convex $n$-gon $P$ and its decomposition into parts:
$P=\cup_{j=1}^n P_j$ such that $z\in P_j$
if and only if $z$ is closer to the $j$-th side of $P$, than to others.
Then we investigate the set $\cup _{j=1}^n\partial P_j$. It is a graph
and after removing from it the boundary points  of $P$ and open, i.e. without endpoints, edges which have nonempty intersection with $\partial P$ we obtain its subgraph; denote it by $Gr(P)$.
For example, if we take as $P$  a rectangle $R=[-k,k]\times [-1,1]$,
then $Gr(P)$ consists of the segment $\Sigma_0$ (see Section~\ref{rectht}).
If $P$ admits an inscribed circle (a circle which is tangent to all sides of $P$),
then  $Gr(P)$ consists of the point which is the center of the circle.

Lemma~\ref{segm} immediately implies the following result.

\begin{thm}\label{ngon}
The maximum of $2d_P(u)/r_P(u)$, $u\in P$, is attained at some point of the graph $Gr(P)$.
\end{thm}

Analyzing some concrete domains and classes of domains (see the examples below), we can suggest the following conjecture.

\begin{conjecture}\label{convertex}
The maximum of $2d_P(u)/r_P(u)$, $u\in P$, is attained at a vertex of the graph $Gr(P)$.
\end{conjecture}

\begin{rem}\label{infin}
We note that the statement of Theorem~\ref{ngon} is also valid for  most unbounded convex polygonal domains $P$, excluding those with degenerate (empty) $Gr(P)$ such as strips and  infinite sectors. Actually, if $P$ is unbounded, then $\partial P$ contains two rays, $L_1$ and $L_2$. If the rays are not parallel, for simplicity, we can assume that their continuations intersect at the origin. Consider the set $\mathcal{A}$ of points in $P$ such that maximal disks centered at these points and contained in $P$ touch both the rays; it is also a ray coming from a point $z_0$. Then applying the mappings $z\mapsto \alpha z$, $\alpha>0$, we show as in the proof of Lemma~\ref{segm} that when approaching the origin by the set $\mathcal{A}$ the value of  $2d_P(u)/r_P(u)$  increases. Therefore the maximal value of $2d_P(u)/r_P(u)$, $u\in \mathcal{A}$, is attained at the vertex $z_0$ of $Gr(P)$. If $L_1$ and $L_2$ are parallel, then the ray $\mathcal{A}$  is parallel to them and we can apply shifts $z\mapsto z+h$ instead of the mappings $z\mapsto \alpha z$, $\alpha>0$.
\end{rem}


\begin{nonsec}{\bf Examples.}\smallskip

1) {\bf Convex sector} $S_\gamma:=\{z\in\mathbb{C}: 0<\arg z<\gamma\}$, $0<\gamma<\pi$;
The upper estimate for the value of $2d_{S_\gamma}(\cdot)/r_{S_\gamma}(\cdot)$ is attained at every point of its bisector.
 The function $g(z) = (z^{{\pi}/{\gamma}}-i)/(z^{{\pi}/{\gamma}}+i)$
  maps the sector $S_\gamma$ onto the unit disk $\mathbb{B}^2$ so
  that the point $e^{{i\gamma}/{2}}$ goes to the origin.
  Using the formula \eqref{crad1}, we obtain that
$$
r_{S_\gamma}\bigl(e^{{i\gamma}/{2}}\bigr)=
\frac{1}{\bigl|g'(e^{{i\gamma}/{2}})\bigr|} = \frac{2\gamma}{\pi}.
$$
Therefore,
$$\frac{2d_{S_\gamma}\bigl(e^{{i\gamma}/{2}}\bigr)}
{r_{S_\gamma}\bigl(e^{{i\gamma}/{2}}\bigr)} = \frac{\pi\sin(\gamma/2)}{\gamma}.$$
If $\gamma$ tends to $\pi-$, then $S_\gamma$ tends to the half-plane and ${2d_{S_\gamma}\bigl(e^{{i\gamma}/{2}}\bigr)}/
{r_{S_\gamma}\bigl(e^{{i\gamma}/{2}}\bigr)} \rightarrow 1$. If we allow $\gamma \to 0+$, then  ${2d_{S_\gamma}\bigl(e^{{i\gamma}/{2}}\bigr)}/
{r_{S_\gamma}\bigl(e^{{i\gamma}/{2}}\bigr)}$ tends to $\pi/2$; this constant corresponds to the case of the strip $\{z:\left|\Im{z}\right|<\pi/{2}\}$.

\medskip

2)  {\bf Isosceles triangles.} For an arbitrary triangle $\Delta$, the upper bound of $2d_{\Delta}(\cdot)/r_{\Delta}(\cdot)$ is attained at the center of its inscribed circle.
Here we describe the situation for the case of an isosceles triangle $\Delta_\alpha$ with vertices $1$, $\pm i\cot(\alpha\pi)$ and angles $(1-2\alpha)\pi$, $\alpha\pi$ and $\alpha\pi$, where $\alpha\in(0,1/2)$.
The conformal mapping of the right half plane onto $\Delta_\alpha$ is given by
$$
G_\alpha(z)=\frac{2}{\sqrt{\pi}}\,\frac{\Gamma(1-\alpha)}{\Gamma(1/2-\alpha)}
\int_0^z\frac{dt}{(1+t^2)^{1-\alpha}}
=\frac{2z}{\sqrt{\pi}}\,\frac{\Gamma(1-\alpha)}{\Gamma(1/2-\alpha)}\,{}_2F_1
\Bigl(\frac{1}{2},1-\alpha;\frac{3}{2};-z^2\Bigr).
$$
Here ${}_2F_1$ is the Gaussian hypergeometric function \cite[Ch 1]{avv}. The center of the inscribed circle for $\Delta_\alpha$ is at the point $(1-\tan^2(\alpha\pi/2))/2$ and its preimage $x=x(\alpha)$ under the mapping $G_\alpha$  can be found from the equation
$$
\frac{2x}{\sqrt{\pi}}\,\,\frac{\Gamma(1-\alpha)}{\Gamma(1/2-\alpha)}\,{}_2F_1
\Bigl(\frac{1}{2},1-\alpha;\frac{3}{2};-x^2\Bigr)\,=\frac{1}{2}\,
\Bigl(1-\tan^2\frac{\alpha\pi}{2}\Bigr).
$$
Then the maximal value of
$2d_{\Delta_\alpha}(\cdot)/r_{\Delta_\alpha}(\cdot)$ is equal to
$$
M_1(\alpha):=\frac{\sqrt{\pi}(1-\tan^2({\alpha\pi/}{2}))
\Gamma(1/2-\alpha)(1+x(\alpha)^2)^{1-\alpha}}{4\Gamma(1-\alpha)x(\alpha)}.
$$
From the graph of $M_1(\alpha)$ (see Fig.~\ref{figM3} and Remark~\ref{M123}) we can see  
that the maximal value of
$M_1(\alpha)$ is attained at the point $\alpha_0=1/3$ corresponding to the case of regular triangle; possibly this can be
proved strictly analytically.  The corresponding value is
$$
M_1(\alpha_0)= \frac{2^{1/3} \sqrt{\pi}\Gamma(1/6)}{3\sqrt{3}\Gamma(2/3)}=1.76663875\ldots.
$$
In Table \ref{Tab:2}, we give some values of
$M_1(\alpha)$, where $\alpha\in(0,1/2)$.
\begin{table}[!htb]
\centering
\caption{Some values of $M_1(\alpha)$, where $\alpha\in(0,1/2)$} \label{Tab:2}
\scalebox{0.85}{
 \begin{tabular}{rr}
\hline\noalign{\smallskip}
 $\alpha$\phantom{A} & $M_1(\alpha)$\phantom{AAA} \\
\noalign{\smallskip}\hline\noalign{\smallskip}\noalign{\smallskip}\noalign{\smallskip}
$\to 0+\phantom{a}$ & $\to ({\pi}/{2})=1.5707963267\ldots$ \\[2mm]
${1}/{12}\phantom{A}$ & $1.6527823736\ldots$ \\[2mm]
${1}/{6}\phantom{A}$  & $1.7147087445\ldots$ \\[2mm]
${1}/{4}\phantom{A}$ & $1.7534053628\ldots$ \\[2mm]
${5}/{12}\phantom{A}$  & $1.7531525258\ldots$ \\[2mm]
$\to (1/2)-$  & \phantom{AAA}$\to ({\pi}/{2}) \coth(\pi/2)=1.7126885749\ldots$ \\
\noalign{\smallskip}\hline
\end{tabular}}
\end{table}

We note that $\alpha\to 0+$ gives the case of a strip and $\alpha\to (1/2)-$ gives the case of a half-strip, see the next example.

\smallskip

3)  {\bf Half-strip.}
$\Omega_a:=\{z: |\Re z|< a$\, {\rm and}\, $\Im z>0, \, 0<a\in \mathbb{R} \}$;
The upper estimate $2d_{\Omega_a}(\cdot)/r_{\Omega_a}(\cdot)$ is attained
at the intersection point of the bisectors of the two angles at its  two
finite vertices.
The function mapping the half-strip $\Omega_a$ onto the upper half-plane is $g(z)=\sin (\pi z/2a)$. The intersection point of the bisectors is $ia$. It follows from \eqref{crad2} that
\begin{equation*}
  r_{\Omega_a}(ia)=\frac{2\Im g(ia)}{|g'(ia)|}=\frac{2\sinh(\pi/2)}{(\pi/2a)\cosh(\pi/2)}=\frac{4a}{\pi}\tanh(\pi/2).
\end{equation*}
Therefore,
\begin{equation*}
  \frac{2d_{\Omega_a}(ia)}{r_{\Omega_a}(ia)}=\frac{\pi}{2}\coth(\pi/2)=:M_0\approx 1.712689.
\end{equation*}
\medskip


4)  {\bf Arbitrary convex triangle with a vertex at infinity.}
 Let $\Lambda_\alpha$ be a symmetric triangle bounded with the segment
$[-i,i]$ and two rays going from the points $\pm i$ and forming the
angles $\alpha\pi$ with the segment, where $\alpha\in(1/2,1)$.
The upper estimate for $2d_{\Lambda_\alpha}(\cdot)/r_{\Lambda_\alpha}(\cdot)$
is attained at the intersection point of the bisectors of the two angles at
its two finite vertices.
The function $P_\alpha$ mapping the right half plane onto $\Lambda_\alpha$ has the form
$$
P_\alpha(z)=\frac{2}{\sqrt{\pi}}\,\frac{\Gamma(1/2+\alpha)}{\Gamma(\alpha)}
\int_0^z\frac{{\rm d}t}{(1+t^2)^{1-\alpha}}
=\frac{2z}{\sqrt{\pi}}\,\frac{\Gamma(1/2+\alpha)}{\Gamma(\alpha)}\,{}_2F_1
\Bigl(\frac{1}{2},1-\alpha;\frac{3}{2};-z^2\Bigr).
$$
The intersection point of the bisectors of the angles at the finite vertices is at the point $\tan(\alpha\pi/2)$, therefore we can find $x=x(\alpha)$ which is the preimage of $\tan(\alpha\pi/2)$ under the mapping $P_\alpha$ from the equation
$$
\frac{2x}{\sqrt{\pi}}\,\frac{\Gamma(1/2+\alpha)}{\Gamma(\alpha)}\,{}_2F_1
\Bigl(\frac{1}{2},1-\alpha;\frac{3}{2};-x^2\Bigr)=\tan\left(\frac{\alpha\pi}{2}\right).
$$
Then the maximal value of $2d_{\Lambda_\alpha}(\cdot)/r_{\Lambda_\alpha}(\cdot)$ is equal to
$$
M_2(\alpha):=\frac{\sqrt{\pi}\tan({\alpha\pi/}{2})\Gamma(\alpha)(1+x(\alpha)^2)^{1-\alpha}}
{2\Gamma(1/2+\alpha)x(\alpha)}.
$$
Some values of $M_2(\alpha)$ also are given in Table \ref{Tab:3}:
\begin{table}[!htb]
\centering
\caption{Some values of $M_2(\alpha)$, where $\alpha\in(1/2,1)$} \label{Tab:3}
\scalebox{0.85}{
 \begin{tabular}{lc}
\hline\noalign{\smallskip}
 $\alpha$ & $M_2(\alpha)$ \\
\noalign{\smallskip}\hline\noalign{\smallskip}\noalign{\smallskip}\noalign{\smallskip}
${2}/{3}$ & \phantom{A}$1.554662095759\ldots$ \\[2mm]
${3}/{4}$  & \phantom{A}$1.441224578770\ldots$ \\[2mm]
${5}/{6}$ & \phantom{A}$1.308765658869\ldots$ \\
\noalign{\smallskip}\hline
\end{tabular}}
\end{table}
\begin{rem}\label{M123}
We can combine Examples 2) and 4) above by defining $M_3(\alpha)$ as follows for all $\alpha\in(0,1)$:
\begin{equation*}
  M_3(\alpha):=\left\{
     \begin{array}{rl}
       M_1(\alpha), & \hbox{$0<\alpha<\frac{1}{2}$;} \\[1mm]
       M_0, & \hbox{$\alpha=\frac{1}{2}$;} \\[1mm]
       M_2(\alpha), & \hbox{$\frac{1}{2}<\alpha<1$.}
     \end{array}
   \right.
\end{equation*}
Fig. \ref{figM3} shows the graph of $M_3(\alpha)$ for $\alpha\in(0,1)$. 
\end{rem}
\begin{figure}[ht] \centering
\includegraphics[width=3. in]{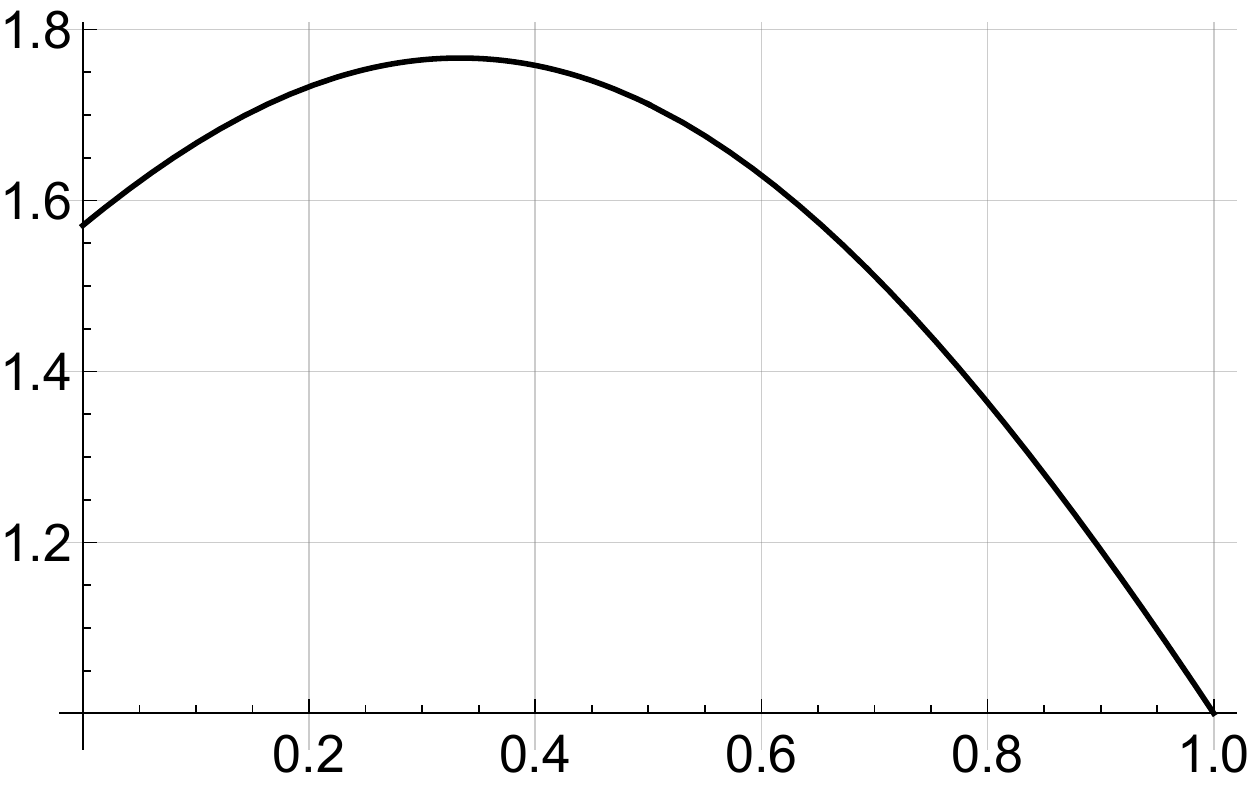}
\caption{The graph of $M_3(\alpha)$, where $\alpha\in(0,1).$}\label{figM3}
\end{figure}


5)  {\bf Rhomb.} Let $\Pi$ be a rhomb with acute angle $\delta$ and length of side $1$. The upper estimate of $2d_{\Pi}(\cdot)/r_{\Pi}(\cdot)$ is attained at its center.
From \cite[Ch.1, \S~1.22]{ps} we have that the conformal radius at its center is equal to
$$
r_\Pi(0)=\frac{4\sqrt{\pi}}{\Gamma(\frac{\delta}{2\pi})\Gamma(\frac{\pi-\delta}{2\pi})}.
$$
Since the distance from the center to the boundary is equal $d_\Pi(0)=(1/2)\sin\delta$, we obtain
$$
\frac{2d_\Pi(0)}{r_\Pi(0)}=\frac{\sin\delta}{4\sqrt{\pi}}\,\Gamma
\left(\frac{\delta}{2\pi}\right)\Gamma\left(\frac{\pi-\delta}{2\pi}\right)
=\frac{\pi\sqrt{\pi}}{2\Gamma\left(\frac{2\pi-\delta}{2\pi}\right)
\Gamma\left(\frac{\pi+\delta}{2\pi}\right)}=:\phi(\delta).
$$
This value tends to
$$
\frac{\Gamma^2(1/4)}{4\sqrt{\pi}}= 1.854074677\ldots,
$$
(the case of a square, cf. \eqref{l0}), as $\delta\to \pi/2$, and to $\pi/2$ (the case of a strip), as $\delta\to 0$.
It is easy to see that the function $\phi(\delta)$ is positive on $(0,\pi/2)$ and
$$
\frac{\phi'(\delta)}{\phi(\delta)}\,=\frac{1}{2\pi}\,\left[\psi\left(\frac{2\pi-\delta}{2\pi}\right)
-\psi\left(\frac{\pi+\delta}{2\pi}\right)\right],
$$
where $\psi(x)=\frac{d}{dx}\log\Gamma(x)$ is the digamma function. Since $\psi(x)$ increases on $(0,+\infty)$ and $(2\pi-\delta)/(2\pi)>(\pi+\delta)/(2\pi)$, $0<\delta<\pi/2$ we see that $\phi(\delta)$ increases on $(0,\pi/2)$.

The graph of $\phi(\delta)$ is given on Fig. \ref{rhomb}.
\begin{figure}[ht] \centering
\includegraphics[width=3. in]{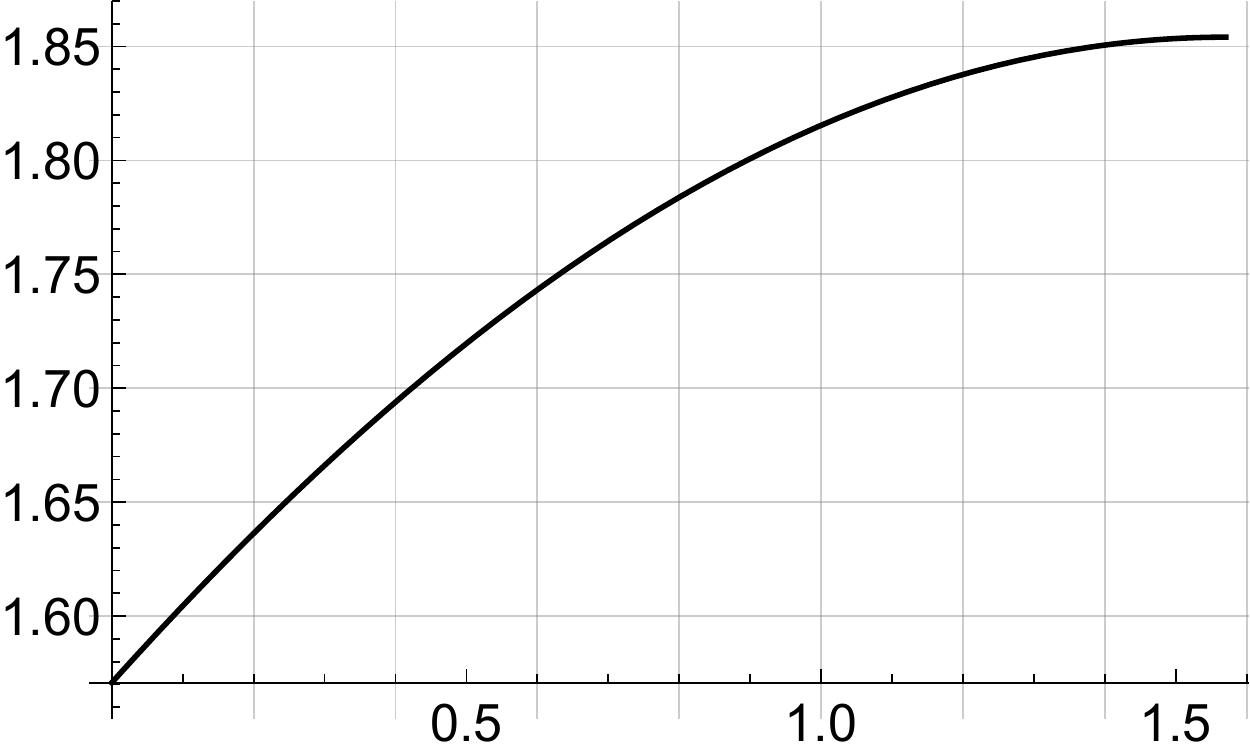}
\caption{The graph of $\phi(\delta)$.}\label{rhomb}
\end{figure}

6)  {\bf Trapezoid.} Let us denote an arbitrary trapezoid by
$\mathrm{T}$.\smallskip

a) For long trapezoids, the maximum of
$2d_\mathrm{T}(\cdot)/r_\mathrm{T}(\,\cdot\,)$ 
is attained at its midsegment, more precisely, between the intersection
points of bisectors of the angles at two adjacent vertices,
lying on different parallel sides. Most likely,  the maximum is attained at one of
the intersection points of its bisectors.
This case requires additional investigation.\smallskip

b) For short trapezoids, the maximum of $2d_\mathrm{T}(\cdot)/r_\mathrm{T}(\cdot)$ 
is attained at the segment joining the intersection points of
 bisectors of the angles at two adjacent vertices,
 lying on different parallel sides.
This segment is on the bisector of the angle formed by the
straight lines, containing the nonparallel sides of the trapezoid.
Most likely,  the maximum is attained at one of the intersection points of the
bisectors described above. This case also requires additional
investigation.\smallskip

c) The transitional case is when the considered trapezoid
has an inscribed circle. Then the upper estimate is attained at
its center.\smallskip

7)  {\bf Arbitrary convex quadrilateral.}  The situations is close to the case of trapezoids. The only difference that, in the non-trapezoidal case, instead of a part of midsegment we consider the segment of bisector,  which form the straight lines, containing the corresponding pair of nonparallel sides.\smallskip

8)  {\bf Regular $n$-gon.} Let $\mathrm{N}$ be a regular $n$-gon with the side length $a$. The upper estimate of $2d_\mathrm{N}(\cdot)/r_\mathrm{N}(\cdot)$ is attained at its center.
Again, from \cite[Ch.1, \S~1.22]{ps} we have that the conformal radius
at its center is equal to
$$
r_\mathrm{N}(0)=\dfrac{na\Gamma\left(1-\frac{1}{n}\right)}{2^{1-\frac{2}{n}}
\Gamma\left(\frac{1}{2}\right)\Gamma\left(\frac{1}{2}-\frac{1}{n}\right)}.
$$
The distance from the center to the boundary is equal to the radius of the inscribed circle:
$$d_\mathrm{N}(0) = \dfrac{a}{2\tan\frac{\pi}{n}}.$$
Thus we obtain

$$\dfrac{2d_\mathrm{N}(0)}{r_\mathrm{N}(0)} = \dfrac{2^{1-\frac{2}{n}}\Gamma\left(\frac{1}{2}\right)\Gamma\left(\frac{1}{2}
-\frac{1}{n}\right)}{n\tan\frac{\pi}{n}\Gamma\left(1-\frac{1}{n}\right)}
={\textstyle\frac{1}{n}}\,2^{1-\frac{2}{n}}B({\textstyle\frac{1}{n},\frac{1}{2}}).
$$
This value grows as $n$ increases  and tends to $2$, as $n$ tends to infinity  (the case of a circle).
For $n$ equal to $3$, $4$ and $6$ this formula gives values for equilateral triangle, square and regular hexagon, respectively.\smallskip

9)  {\bf Polygons having inscribed circles.}  The upper estimate is attained at the center of the inscribed circle.\medskip



Now we will estimate the ratio $2d_{D_1}(\cdot)/r_{D_1}(\cdot)$ and $2d_{D_2}(\cdot)/r_{D_2}(\cdot)$  for non-polygonal domains described in Examples~\ref{ex1} and~\ref{ex2}, respectively.\smallskip


  10)  {\bf The domain $D_1$} which is  the image of the unit disk under the mapping $w=\phi_1(z)=\sqrt{1+z}-1$ 
  is convex.  From Lemma~\ref{segm} it follows that  the maximum of the ratio $2d_{D_1}(w)/r_{D_1}(w)$, $w\in D_1$ is attained at some point $w$ which is real, i.e. $-1<w<\sqrt{2}-1$.
  Now we will find  $d_{D_1}(w)$ for such $w$.

  Denote $t=w+1$, $0<t<\sqrt{2}$. Then
  $$\sqrt{1+e^{i\theta}}=e^{i\theta/4}\sqrt{2\cos(\theta/2)}, \quad -\pi/2\le \theta\le \pi/2
  $$ and
  $$
A(\theta):=|\phi_1(e^{i\theta})-w|^2=|\sqrt{1+e^{i\theta}}-t|^2
=t^2-2\sqrt{2}t\cos(\theta/4)\sqrt{\cos(\theta/2)}+2\cos(\theta/2).
  $$
If we put $\tau=\cos(\theta/4)$, then we obtain $A(\theta)=f(\tau)$, where
$$
f(\tau)=t^2-2\sqrt{2}t\tau\sqrt{2\tau^2-1}+2(2\tau^2-1),\quad \sqrt{2}/2\le|\tau|\le 1.
$$
Analyzing $f(\tau)$, we find that for $2\sqrt{2}/3\le t<\sqrt{2}$ we have $f(\tau)\ge f(1)=(\sqrt{2}-t)^2$, therefore,  the distance from $w$ to the boundary of $D_1$ equals $d_{D_1}(w)=\sqrt{2}-t$. 
If $0<t<2\sqrt{2}/3$, then  $f(\tau)$ attains its minimum at the point $\tau_0$ such that
\begin{equation}\label{tau0}
\tau_0^2=\frac{1}{4}\left(1+\frac{1}{\sqrt{1-t^2}}\right)
\end{equation}
and then $d_{D_1}(w)=\sqrt{\sqrt{1-t^2}-(1-t^2)}$. Therefore,
\begin{equation}\label{dd1}
d_{D_1}(w)=\left\{
             \begin{array}{ll}
                \sqrt{\sqrt{1-t^2}-(1-t^2)},& 0<t<2\sqrt{2}/3,\\
               \sqrt{2}-t, & 2\sqrt{2}/3\le t<\sqrt{2},
             \end{array}
           \right.
\end{equation}
where $t=1+w$.

Let $\tau_0=\cos(\theta_0/4)$.  Analyzing \eqref{tau0}, we see
that when $t$ increases from $0$ to  $2\sqrt{2}/3$ the point $\theta_0$
decreases from $\pi$ to $0$. Therefore, for every boundary point
$\phi_1(\theta)$, $0< \theta<\pi$,  there is $t\in[0,2\sqrt{2}/3]$
such that the maximal circle in $D_1$ centered at the point
$w=t-1$ touch $\phi_1(\theta)$. Because of the symmetry of $D_1$
with respect to the real axis, a similar fact is valid for
$-\pi< \theta<0$. Then from Lemma~\ref{segm} we conclude that
the maximal value of $2d_{D_1}(w)/r_{D_1}(w)$, $w\in D_1$, is
attained at some point of the real axis.

For the conformal radius at the points $w=\phi_1(z)$, $-1<z<1$, we have $r_{D_1}(w)=(1/2)(1-z)^{1/2}(1+z)$. Taking into account that $z=(w+1)^2-1=t^2-1$ we have
\begin{equation}\label{rd1}
r_{D_1}(w)=(1/2)t(2-t^2).
\end{equation}

Simple analysis of the function $2d_{D_1}(w)/r_{D_1}(w)$ with the help of \eqref{dd1} and \eqref{rd1} shows that it has a unique maximum at the point $w_0=\sqrt{1-u_0^2}-1=-0.109718\ldots$, where $u_0=0.45541\ldots$ is  a unique root of the cubic equation $4u^3+3u^2-1=0$.
The maximum equals $1.85318\ldots$. Therefore, in $D_1$ we have the sharp inequality
\begin{equation*}
\frac{2d_{D_1}(w)}{r_{D_1}(w)}\le 1.85318\ldots, \quad w\in D_1.
\end{equation*}

 11)  {\bf The domain $D_2$} which is  the image of the unit disk under the mapping $w=\phi_2(z)=z+\sqrt{1+z^2}-1$. 
 We  will also find the maximum of the ratio $2d_{D_2}(w)/r_{D_2}(w)$ but, because of the fact that this domain is not convex, we will confine ourselves to considering the case of real $w$.

For real $w=\phi_2(z)$, we have
$$
r_{D_2}(w)=\frac{z+\sqrt{1+z^2}}{\sqrt{1+z^2}}(1-z^2).
$$
Since the disk $\{|w-(\sqrt{2}-1)|<1\}$ is a subset of $D_2$ and it contains two its boundary points $\sqrt{2}-2$ and  $\sqrt{2}$, we have for $\sqrt{2}-2<w<\sqrt{2}$:
$$
d_{D_2}(w)=\left\{
             \begin{array}{ll}
               w-(\sqrt{2}-2),& \sqrt{2}-2<w<\sqrt{2}-1,\\[2mm]
               \sqrt{2}-w, & \sqrt{2}-1\le w<\sqrt{2}.
             \end{array}
           \right.
$$
If we denote $t=w+1$, then
$$
r_{D_2}(w)=\frac{6t^2-t^4-1}{2(t^2+1)},
$$
$$
d_{D_2}(w)=\left\{
             \begin{array}{ll}
               t-(\sqrt{2}-1),& \sqrt{2}-1<t<\sqrt{2},\\[2mm]
               \sqrt{2}+1-t, & \sqrt{2}\le t<\sqrt{2}+1
             \end{array}
           \right.
$$
and
\begin{equation}\label{drd2}
\frac{2d_{D_2}(w)}{r_{D_2}(w)}=B(t):=\left\{
             \begin{array}{ll}
               \displaystyle\frac{(t-(\sqrt{2}-1))(t^2+1)}{6t^2-t^4-1},& \sqrt{2}-1<t<\sqrt{2},\\[4mm]
               \displaystyle\frac{((\sqrt{2}+1)-t)(t^2+1)}{6t^2-t^4-1}, & \sqrt{2}\le t<\sqrt{2}+1,
             \end{array}
           \right.
\end{equation}
From \eqref{drd2} it is easy to see that $B(t)$ increases on $[\sqrt{2}-1,\sqrt{2}]$ and decreases on $[\sqrt{2},\sqrt{2}+1]$. Therefore, the maximum is attained at $t=\sqrt{2}$ and is equal to $B(\sqrt{2})=12/7=1.71429\ldots$.

\end{nonsec}

\noindent
{\bf Acknowledgment.} Funding for the second author was provided by Turku University Foundation. The work of the third author performed under the development program of the Volga Region Mathematical Center (agreement no. 075-02-2022-882).



\bibliographystyle{siamplain}

\begin{thebibliography}{CHKV}



\bibitem[AF]{af}
\textsc{M.~J.~Ablowitz and A.~S.~Fokas,}
\emph{Complex variables: introduction and applications.} Second edition.
Cambridge Texts in Applied Mathematics. Cambridge University Press, Cambridge, 2003. xii+647 pp.


\bibitem[A]{a}
\textsc{L.~V.~Ahlfors,} \emph{M\"obius transformations in several dimensions.}
Ordway Professorship Lectures in Mathematics. University of Minnesota,
School of Mathematics, Minneapolis, Minn., 1981. ii+150 pp.

\bibitem[AVV]{avv}
\textsc{G.\,D. Anderson, M.\,K. Vamanamurthy, and M.~Vuorinen},
\emph{Conformal invariants, inequalities, and quasiconformal maps.} John Wiley \& Sons, New York, 1997.  xxviii+505 pp


\bibitem[BF]{bf}
\textsc{C.~Bandle and M.~Flucher},
\emph{Harmonic radius and concentration of energy; hyperbolic radius and Liouville's equations $\Delta U = e^U$ and $\Delta U = U^{(n + 2)/(n - 2)}$.} SIAM Rev. 38 (1996), no. 2, 191--238.

\bibitem[BKN]{bkn} \textsc{A.~D.~Baranov, I.~R.~Kayumov, and S.~R.~Nasyrov,}
\emph{On Bloch seminorm of finite Blaschke products in the unit disk.}
J. Math. Anal. Appl. 509 (2022), no. 2, Paper No. 125983, 10 pp.

\bibitem[B]{b}
\textsc{A.~F.~Beardon},  \emph{The geometry of discrete groups.} Graduate Texts in Mathematics, 91. Springer-Verlag, New York, 1983. xii+337 pp.

\bibitem[BM]{bm}
\textsc{A.~F.~Beardon and D.~Minda,}
\emph{The hyperbolic metric and geometric
function theory.} Proc. International Workshop on Quasiconformal Mappings
and their Applications (IWQCMA05),
eds. S.~Ponnusamy, T.~Sugawa and M.~Vuorinen (2006), 9-56.
\bibitem[DNRV]{dnrv}
{\sc D.~Dautova, S.~Nasyrov, O.~Rainio, and  M.~Vuorinen,}
\emph{Metrics and quasimetrics induced by point pair function.}
\href{http://arxiv.org/abs/2202.08518}{arXiv:2202.08518}, 22 pp.

\bibitem[D]{duren}
\textsc{P.~L.~Duren,} \emph{Univalent functions.} Springer-Verlag, New York inc., 1983.



\bibitem[GOL]{gol}
\textsc{G.~M.~Goluzin,} \emph{Geometric theory of functions of a complex variable.}
Translations of Mathematical Monographs,
Vol. 26 American Mathematical Society, Providence, R.I. 1969 vi+676 pp.

\bibitem[HKV]{hkv}
\textsc{P.~Hariri, R.~Kl\'en, and M.~Vuorinen},  \emph{Conformally invariant metrics and quasiconformal mappings.} Springer Monographs in Mathematics,
Springer, Berlin, 2020. xix+502 pp.


\bibitem[H]{hayman}
\textsc{W.~K.~Hayman,} \emph{Multivalent functions.}  Cambridge Tracts in Mathematics, Series Number~110. 2nd Ed. Cambridge Univ. Press., 1994.


\bibitem[KT]{kt2021}
{\textsc{R.~Kargar and L.~Trojnar-Spelina,}
\emph{Starlike functions associated with the generalized Koebe function.} Anal. Math. Phys. 11 (2021), no. 4, Paper No. 146, 26 pp.}



%



\bibitem [P]{p}
{\textsc{A. Papadopoulos,}} \emph{ Metric spaces, convexity and non-positive curvature.} Second edition. IRMA Lectures in Mathematics and
Theoretical Physics, 6. European Mathematical Society (EMS), Z\"urich, 2014. xii+309 pp.
\bibitem [PS]{ps}
{\textsc{G.~P\'olya and G.~Szeg\"o,}}  \emph{Isoperimetric Inequalities in
Mathematical Physics.} Annals of Mathematics Studies,
No. 27 Princeton University Press, Princeton, N. J., 1951. xvi+279 pp.

\bibitem [R]{r}
{\textsc{O. Rainio,}}
\emph{Intrinsic metrics under conformal and quasiregular mappings.}
arXiv:2103.04397
\bibitem [RV]{rv}
{\textsc{O.~Rainio and  M.~Vuorinen,}}
\emph{Triangular ratio metric under quasiconformal mappings in sector domains}.
Comput. Methods Funct. Theory (2022). https://doi.org/10.1007/s40315-022-00447-3,
\href{http://arxiv.org/abs/2005.11990}{arXiv:2005.11990}.

\bibitem [So]{sol}
{\textsc{A.~Yu.~Solynin,}}
\emph{Minimization of the conformal radius under circular restriction
of the domain.} (Russian) Zap. Nauchn. Sem. S.-Peterburg. Otdel. Mat. Inst. Steklov.
(POMI) 254 (1998), Anal. Teor. Chisel i Teor. Funkts. 15, 145--164, 247;
translation in J. Math. Sci. (New York) 105 (2001), no. 4, 2220--2234.







\end{thebibliography}


\end{document}